\documentclass[12pt]{article}
\usepackage{amssymb, amsmath, amsthm, epsf, epsfig, color}

\newlength{\originalbase}
\newcommand{\spacing}[1]{\setlength{\baselineskip}{#1\originalbase}}
\newcommand{\R}{\mathbb{R}}
\renewcommand{\P}{\mathbb{P}}
\newcommand{\E}{\mathbb{E}}

\begin{document}
\spacing{1.5}

\newtheorem{theorem}{Theorem}[section]
\newtheorem{claim}[]{Claim}
\newtheorem{prop}[theorem]{Proposition}
\newtheorem{remark}[theorem]{Remark}
\newtheorem{lemma}[theorem]{Lemma}
\newtheorem{corollary}[theorem]{Corollary}
\newtheorem{guess}[theorem]{Conjecture}
\newtheorem{conjecture}[theorem]{Conjecture}

\title{Catching the Drunk Robber on a Graph}

\author{Natasha Komarov\thanks{Department of Mathematics, Dartmouth,
Hanover NH 03755-3551, USA; nkom@dartmouth.edu.}
\, and Peter Winkler\thanks{Department of Mathematics, Dartmouth,
Hanover NH 03755-3551, USA; peter.winkler@dartmouth.edu.  Research
supported by NSF grant DMS-0901475.}}

\maketitle

\begin{abstract}
We show that the expected time for a smart ``cop'' to catch a drunk
``robber'' on an $n$-vertex graph is at most $n + {\rm o}(n)$.
More precisely, let $G$ be a simple, connected, undirected graph
with distinguished points $u$ and $v$ among its $n$ vertices.
A cop begins at $u$ and a robber at $v$; they move alternately
from vertex to adjacent vertex.  The robber moves
randomly, according to a simple random walk on $G$; the cop sees
all and moves as she wishes, with the object of ``capturing'' the
robber---that is, occupying the same vertex---in least expected
time.  We show that the cop succeeds in expected time no more than
$n + {\rm o}(n)$. Since there are graphs in which capture time is at least $n - o(n)$, this is roughly best possible. We note also that no function of the diameter can be a bound on capture time.
\end{abstract}

\section{Introduction}

The game of cops and robbers on graphs was introduced independently by
Quilliot \cite{Q} and Nowakowski and Winkler \cite{NW}, and has generated
a great deal of study; see, e.g., \cite{1,2,3,4}.  In the original
formulation a cop and robber move alternately and deliberately, with
full information, from vertex to adjacent vertex on a graph $G$, with
the cop trying to capture the robber and the robber trying to elude
the cop. In this work, all graphs are assumed to be connected, simple (no loops or multiple edges) and undirected. A graph is said to be ``cop-win'' if there is a vertex $u$ such
that for every $v$, the cop beginning at $u$ can capture the robber
beginning at $v$.

In addition to their obvious role in pursuit games, cop-win graphs (which are also
known as ``dismantlable'' graphs) have appeared in diverse places
including statistical physics \cite{BW2}.  In the present work, we
consider a variation suggested \cite{McG} by Ross Churchley of the University of
Victoria, in which the robber is no longer in control of his fate; instead,
at each step he moves to a neighboring vertex chosen uniformly at random.
We may therefore imagine that the robber is in fact a drunk---one who is too far gone to have an objective.

On any graph, the drunk will be caught with probability one, even by a cop who
oscillates on an edge, or moves about randomly; indeed, by any cop who isn't actively
trying to lose.  The only issue is: how long does it take?  The lazy cop will win
in expected time at most $4n^3/27$ (plus lower-order terms), since that is the maximum possible expected
hitting time for a random walk on an $n$-vertex graph \cite{BW};
the same bound applies to the random cop \cite{CTW}.  It is easy to
see that the greedy cop who merely moves toward the drunk at every step can
achieve O$(n^2)$; in fact, we
will show that the greedy cop cannot in general do better.  Our smart
cop, however, gets her man in expected time $n + {\rm o}(n)$.  Note
that when the adversaries play on a lollipop graph consisting of a clique of size $cn^{1/3}$ (for some constant $c \in \R$) with a path of length $n - cn^{1/3}$ attached at one end, with the drunk starting in the clique and the cop starting at the opposite endpoint of the path, the expected capture time will be $n - \Theta (n^{1/3}) = n - {\rm o}(n)$, and we conjecture that this is worst possible.

\section{Preliminaries}
\label{prelims}

In this variation, a ``move'' (as in chess) will consist of a step by the cop followed by a (uniformly random) step by the drunk. Capture or ``arrest'' takes place when the cop lands on the drunk's vertex or vice-versa,
and the capture time $T$ is the number of the move at which this takes place.

Let us consider some examples.  (1) Suppose $G$ is the path $P_n$ on $n$ vertices,
with $u$ and $v$ its endpoints.  Then the cop will (using any of the
algorithms we consider later) move along the path until she reaches the drunk; this
will take expected time about $n - \sqrt{n}$ since a random walk
on a path will on average progress about distance $\sqrt{t}$ in time $t$.

(2) Let $G$ be the complete balanced bipartite graph $K_{\lfloor n/2 \rfloor, \lceil n/2 \rceil}$,
with the cop and the drunk beginning on the same side.  Then the poor cop will find herself always
moving to the opposite side from her quarry until, finally, he accidently runs into her;
since the latter event occurs with probability about $2/n$, arrest takes on average $n/2$ steps.

The reader may feel with some justification that we are being unrealistic in not allowing
the cop to stay put; in example (2), sitting for one move would enable her to catch the drunk on
the next move.  Ultimately, we force the cop to move at each step in order to hold her to
the same constraints as her quarry's, and because it gives us the strongest results.  Our
bounds still apply when the cop, the drunk, or both are allowed to stay put on any move.

Even when the cop is permitted to idle, she cannot expect to catch the drunk in time bounded
by a function of the diameter of $G$.  Example (3), let $G$ be the incidence graph of
a projective plane of order $n$. A projective plane $P$ of order $n$ is a collection of objects called ``points'' and sets of points called ``lines'' satisfying the following conditions:
\begin{enumerate}
\item \label{2points1line} Two points determine a unique line.
\item \label{2lines1point} Two lines intersect in a unique point.
\item \label{n+1points} Every line consists of exactly $n+1$ distinct points.
\item \label{n+1lines} Every point lies on exactly $n+1$ distinct lines.

Furthermore \cite{MHall},
\item $P$ contains exactly $n^2 + n + 1$ distinct points.
\item $P$ contains exactly $n^2 + n + 1$ distinct lines.
\end{enumerate}

Projective planes of order $n$ are known to exist for $n = p^a$ for any prime number $p$ and positive integer $a$ \cite{primepowerPP}. The incidence graph $G$ of $P$ is therefore a graph with $2(n^2+n+1)$ vertices, with adjacency relation $u \sim v$ if $u$ is a point in $P$ and $v$ is a line that goes through $u$, or vice versa. Such graphs have bounded diameter but unbounded expected capture time:
\begin{quote}
{\bf Claim 1.} $diam(G) = 3$.
\begin{proof}
Let $a,b$ be two points in $P$. By condition (\ref{2points1line}) above, $a$ and $b$ both lie on a common line, so $d(a,b) = 2$. If $a,b$ are instead two lines in $P$, then condition (\ref{2lines1point}) says that $a$ and $b$ intersect at a common point. Finally, if $a$ is a point and $b$ is a line in $P$, then either $a$ lies on $b$ and so $d(a,b) = 1$ or there is another point, $c$, which does lie on $b$. But by the previous argument, $d(a,c) = 2$ and so $d(a,b) = 3$.
\end{proof}

{\bf Claim 2.} The girth of $G$ is 6. 
\begin{proof}
Note that $G$ has no odd cycles by the independence of the set of points (and respectively, set of lines). Now assume for sake of contradiction that $G$ contains a cycle of length $4$. Then there are two points $p_1, p_2$ and two lines $\ell_1, \ell_2$ such that $p_1, \ell_1, p_2, \ell_2, p_1$ forms a cycle. But this contradicts condition (\ref{2lines1point}) since $\ell_1$ and $\ell_2$ must intersect in $p_1$ as well as $p_2$. 
\end{proof}

{\bf Claim 3.} $G$ is regular of degree $r = n+1 \approx \displaystyle \sqrt{|V(G)|/2}$. 
\begin{proof}
By conditions (\ref{n+1points}) and  (\ref{n+1lines}).
\end{proof}

{\bf Claim 4.} The expected capture time on $G$ is at least $r$. 
\begin{proof}
When the cop gets to distance 2 of the drunk, he has only one bad move out of $r$; the rest keep him at distance at least 2. (Similarly, if the cop gets to distance 1, bypassing ever being at distance 2, the drunk still has only one bad move out of $r$, the rest of which keep him at distance 1.) Hence the cop's expected capture time cannot be any lower than $r$ (the expected number of independent Bernoulli trials, each with success probability $1/r$, until success is achieved).
\end{proof}
\end{quote}


On the other hand, it is not hard to verify that on any regular graph, the greedy cop---who
minimizes her distance to the drunk at each move---wins in expected time at most linear in $n$.  If $G$ is regular of
degree $r$, its diameter cannot exceed $\frac{3n - r - 3}{r+1}$ \cite{So}.  Since the drunk will
step toward the cop with probability at least $1/r$ at each move, resulting (after her
response) in a decrease of 2 in their distance, the expected capture time is bounded
by $r\cdot diam(G)/2 < 3n/2$.

The linear bound also holds on trees. To see this, we proceed by induction on the size of the tree, $n$. When $n=2$, the capture time is clearly less than $n$ (since the drunk will run into the cop on his first move). Now suppose that on any tree with $t < n$ vertices, the expected capture time is at most $t$, and let $T$ be a tree on $n$ vertices, rooted at $c_0$ (the cop's initial position). For all descendants $v$ of $c_0$, let $T_v$ be the subtree of $T$ consisting of $v$ and all of its descendants. So the game begins on $T = T_{c_0}$, and after the first move, since the drunk cannot get ``behind'' the cop without being caught, the game is being played on $T_{c_1}$ where $c_1$ is the cop's position after one step. (Note that by the greedy strategy, $c_1$ is the unique neighbor of $c_0$ which is on the path from $c_0$ to $r_1$, the drunk's position after he takes his first step.) $|V(T_{c_1})| \leq |V(T_{c_0})| - 1 = n-1$ so by the induction hypothesis, the game takes no more than expected time $n-1$ on $T_{c_1}$ and therefore no more than $n$ on $T$.

For general $G$, one can guarantee only that at a given point in time the drunk will step toward
the cop with probability at least $1/\Delta$, giving a bound of order $n^2$ for the greedy cop.
That may appear to be a gross overestimate, especially in light of the special cases discussed above, but a graph with many high-degree vertices can still have large diameter. For example, consider the following graph. 

\begin{figure}[h!]
\centering
\includegraphics[scale=.8]{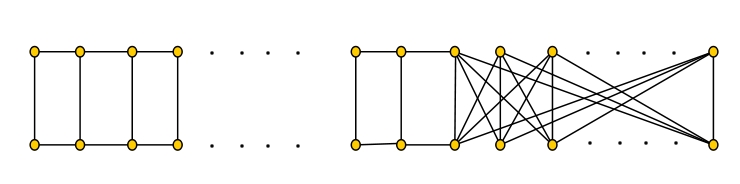}
\caption{The Ladder to the Basement}
\label{ladder}
\end{figure}

The ``ladder'' in this graph consists of two copies of the path $P_{n/4}$ with each pair of corresponding vertices connected by an edge. The ``basement'' consists of a complete bipartite graph, $K_{\lfloor n/4 \rfloor, \lceil n/4 \rceil}$. We begin with the drunk inside the basement, and the cop on the far end of the ladder. While the drunk is meandering inside the basement, the cop---staying true to her goal of minimizing the distance between her and the drunk at each step---is alternating between the two paths. Note that we assume she makes the foolish choice when she is presented with several options by her algorithm.  It takes the drunk $n/4$ moves on average to leave the basement, and each time this occurs, the cop will decrease the distance by 2 by traveling along her current path. Therefore the capture will require an average of about $(n/4)^2/2$ steps.

\section{The Smarter Cop}

\subsection{Intuition}
\label{smartintuition}
As noted in the example of the ``ladder to the basement'' graph in Section~\ref{prelims}, a foolish greedy cop can be foiled by her desire to ``retarget'' too often. That is, since she updates the target vertex (to which she is trying to minimize her distance) at each step, she is made indecisive by an indecisive drunk. One natural solution to this problem would be to walk directly toward the robber's initial position in the basement for several steps before retargeting. Continuing in this way, the cop makes steady progress, ultimately catching the drunk in time less than $n$. 

In general, if a cop and drunk begin at distance $d$ on a graph, and the cop proceeds by retargeting every four steps, then by Lemma~\ref{4lemma} below, it would take $4(4n^{2/3})(d-3)$ moves to get down to distance less than four. Since $d$ can be as large as $n-1$, this would not suffice to yield our promised bound of $n + o(n)$, so the cop must first do something else to get her distance to the drunk down without spending too much time doing so---hence the following four-stage strategy.

	For $i \in [4]$, 
		let $T_i$ be the time spent in Stage $i$ and $D_i$ be the distance between the two players
		at the end of Stage $i$.
	In the first stage, the cop heads directly for the drunk's initial position, $x$, 
		so that $T_1 \le diam(G)$.
	In the meantime, the drunk has gone somewhere else, and so suppose that by the time that the cop
		reaches $x$, the drunk is at $y$. 
	Now we are in Stage 2, and the cop heads for $y$. 
	We show $\E[T_2] = o(n)$. 
	During Stage 3, the cop updates her ``target'' every four steps, and we show that the expected time
		for this stage, $\E[T_3]$, is again bounded by $o(n)$. 
	This stage ends when we are at distance at most three from the drunk. 
	In Stage~4, the cop waits for the drunk to make an error, which happens in expected time at most
		$\Delta$ and results in the capture of the drunk. 
	All together, this cop captures the drunk in expected time $n+o(n)$. 
	We will refer to the progress made by the cop in the first two stages as ``gross progress,'' 
		and in the last two stages as ``fine progress.'' 
	In order to prove the bounds claimed above, it will be beneficial to have a few lemmas.

\subsection{Gross Progress}
	Suppose that the drunk starts on vertex $u$ and the cop starts at $v$. 
	As noted in the set-up of the previous section, in the first stage of the cop's strategy, 
	she is concerned only with getting to $u$ (even if this may not decrease her distance from the drunk at the end of the stage). 
	Clearly the time this takes is equal to $T_1 = d(v,u) \le diam(G)$. 
	We would like to get a bound on $\E[D_1]$, the expected distance between the cop and the drunk 
		at the end of this stage. 
	For that, the following lemma will prove quite useful.

	\begin{lemma}
	\label{keylemma}
		Let $T_{n,t}$ be the distance covered in time $t$ by a random walk on a (connected) graph with 	
			$n$ vertices. 
		Then  $\E[T_{n,t}] < 1 + \sqrt{t}\sqrt{1 + 5 \log n}$.
	\end{lemma}

	\begin{proof}
	Let $p^t(x,y)$ be the probability that a random walk that starts at vertex $x$ will be at vertex $y$ 
		in exactly $t$ steps. 
	The Varopoulous-Carne bound \cite{V}, as formulated in \cite{P}, says 
		$$ p^t(x,y) \le \sqrt{e} \sqrt{\frac{deg(y)}{deg(x)}} \exp \left( -\frac{d(x,y)^2}{2t} \right)$$
	where $d(x,y)$ is the graph distance between the two vertices. 
	Therefore, if we consider the random walk $x_0, x_1, \dots, x_t$ on a graph of size $n$ and let 
		$c \in \R$ be any constant, we have the following bound as a corollary of Varopoulos-Carne:
	
	\begin{eqnarray*}
		\P(d(x_0,x_t) \geq c \sqrt{t} ) & = & \sum_{y:d(x_0,y) \geq c\sqrt{t}} p^t(x_0,y) \\
			& \le & \sum_{y:d(x_0,y) \geq c\sqrt{t}} \sqrt{e} \sqrt{\frac{deg(y)}{deg(x_0)}} \exp \left( -\frac{d(x_0,y)^2}{2t} \right) \\	
			& < & \sum_{y:d(x_0,y) \geq c\sqrt{t}} \sqrt{e} \sqrt{n} \exp \left( -\frac{c^2 t}{2t} \right) \\	
			& < & n^{3/2} \exp\left(\frac{1-c^2}{2} \right)
	\end{eqnarray*} 

	Letting $c = \sqrt{1+5 \log n}$ therefore yields that
		$\P(d(x_0,x_t) \geq \sqrt{1+5 \log n} \, \sqrt{t} ) < \displaystyle \frac{1}{n}$. \\

	Note that $\E[d(x_0,x_t)] \le pn + (1-p)c \sqrt{t}$, where $p = \P(d(x_0,x_t)  \geq c \sqrt{t} )$, 
		so we have

	\begin{eqnarray*}
		\E[d(x_0,x_t)] 	& \le & \frac{1}{n} n  + c \sqrt{t} \\
						& = & 1 + \sqrt{t} \sqrt{1+5 \log n} 
	\end{eqnarray*}

	as desired.
	\end{proof}

	This bound is not tight, but it will be good enough to give us the $o(n)$ bound we seek on $\E[T_2]$.
	
	Recall that $D_1$ is the distance between the two players at the end of Stage 1. 
	Note that this is equivalent to the distance between the drunk's initial position and his position 
		at the end of Stage 1.
	We have the following immediate corollary of Lemma~\ref{keylemma}.

	\begin{corollary}
	\label{keycor}
		$\E[D_1] \le  1 + \sqrt{n} \sqrt{1 + 5\log n}$.
	\end{corollary}

	Now the cop enters Stage 2. 
	We would like to bound $\E[D_2]$.
	Note that this is equivalent to the expected distance traveled by the drunk in Stage 2. 
	
	\begin{corollary}
		$\E[D_2] < (5 \log n)^{3/4} n^{1/4}$
	\end{corollary}
	
	\begin{proof}
		Using Lemma~\ref{keylemma}, Jensen's inequality for concave functions, and Corollary~\ref{keycor}, we get

		\begin{eqnarray*}
			\E[D_2] &\le & \sum_{k=0}^n \P(D_1=k) (1 + \sqrt{k} \sqrt{1 + 5\log n}) \\
					&=& 1 + \sqrt{1 + 5 \log n} \E [\sqrt{D_1}] \\
					& \le & 1 + \sqrt{1 + 5\log n} \sqrt{\E[D_1]} \\
					& \le & 1 + \sqrt{1 + 5\log n} \sqrt{1+\sqrt{n} \sqrt{1 + 5\log n}}\\
					& < & (5 \log n)^{3/4} n^{1/4}
		\end{eqnarray*}
	\end{proof}
	
	Now we are done with the ``gross progress'' that the cop makes in Stages 1 and 2. 
	Note that the total expected time to complete these two stages is bounded by 
	$$\E[T_1] + \E[T_2] \le diam(G) + 1 + \sqrt{n} \sqrt{1 + 5\log n}.$$

\subsection{Fine Progress}
	At the conclusion of stage 2, the cop's approach changes. 
	Now she {\bf retargets} every 4 moves. We make this notion precise in the following manner.

	For each integer $j \geq 1$ let $x_j, y_j$ be the drunk's and cop's positions, respectively, at time
		$j$ (with it being the drunk's turn to move).
	Then in Stage 3, while $d(x_j,y_{j-1}) \geq 4$, for all $j$ of the form $4i+1$ for some $i \geq 0$,
	the cop chooses $x_{4i+1}$ as her {\bf target} and proceeds along a geodesic toward that target for
	the next four steps. 
	Consequently, the cop's target changes every 4 moves, so that for each integer $i \geq 0$, she has
		target $x_{4i+1}$ at times $4i+1, 4i+2$, $4i + 3$, and $4i+4$.
	If at time $j = 4i+1$, $d(x_j,y_{j-1}) < 4$, Stage 3 terminates and the cop's strategy moves into
		Stage 4, which will be described after the following lemma.

	\begin{lemma}
	\label{4lemma}
		Let $G$ be any graph and let $x_0 \in V(G)$ be any vertex in $G$. 
		Let $\{x_0, x_1, x_2, \dots \}$ be any random walk on $G$ beginning at $x_0$. 
		Then $\P(d(x_0,x_4) < 4) \geq 1/s$, where $s = 4n^{2/3}$.
	\end{lemma}

Before we prove this lemma, note that we could not get away with looking at the first three steps of a random walk. That is, we could not get a useful bound for $\P(d(x_0,x_3) < 3)$. Consider the following example: we have a graph $G$ with a vertex $x_0$. Let $A_k$ be the set of vertices at distance $k$ from $x_0$. Suppose that $G$ looks like Figure~\ref{counterexample-3lemma}. That is in $G$, $|A_1| = 1$ and $|A_2| = |A_3| = \displaystyle \frac{n-2}{2}$. Call any step by the random walker that guarantees $d(x_0, x_3) < 3$ a ``stall.'' Then the probability of a stall occurring at the second step is $\displaystyle \frac{1}{(n-2)/2+1} = \frac{2}{n}$, and the probability of a stall occurring at the third step is $\displaystyle \left (1 - \frac{2}{n} \right)\left(\frac{2}{n}\right)$ since for each vertex in $A_2$ and $A_3$, there is one edge on the path toward $x_0$ and $\displaystyle \frac{n-2}{2}$ edges leading farther away from $x_0$. So then $\P(d(x_0,x_3) < 3) = \displaystyle  \frac{2}{n} + \left(1 - \frac{2}{n}\right)\frac{2}{n} < \frac{4}{n}$.

\begin{figure}[h!]
\centering
\includegraphics[scale=.8]{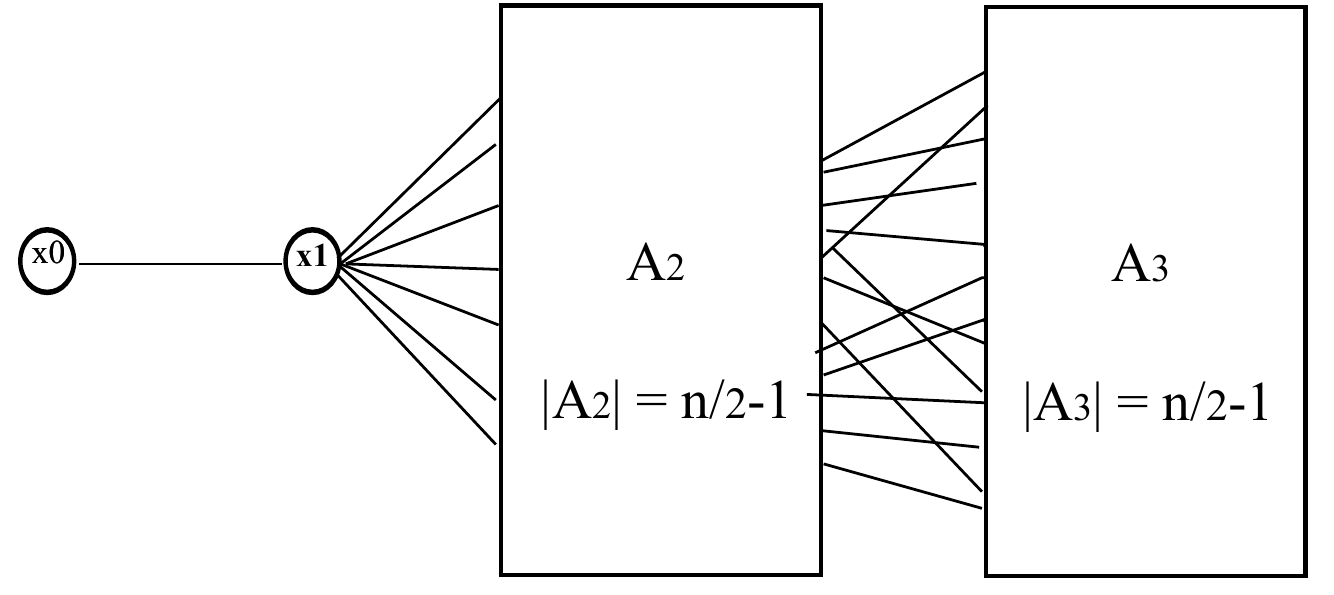}
\caption{At least 4 steps are required to secure a useful bound on a random walk's progress.}
\label{counterexample-3lemma}
\end{figure}

We now return to the proof of Lemma~\ref{4lemma}.

	\begin{proof}
	We proceed by assuming a graph $G$ and a vertex $x_0 \in V(G)$ are such that there is a random walk
		$\{x_0, x_1, \dots \}$ with the property $\P(x_0,x_4) < 4) < 1/s$, and we shall derive a
		contradiction.

	Let $A_k$ be the set of vertices at distance $k$ from $x_0$, and let $a_k = |A_k|$ for all $k$. 
	We adapt the terms {\bf in-degree} and {\bf out-degree} to mean the following: 

		Let $v \in A_k$. 
		Then the in-degree of $v$ is $\deg^-(v) = |N_G(v) \cap A_{k-1}|$ and 
			the out-degree of $v$ is $\deg^+(v) = |N_G(v) \cap A_{k+1}|$. 
	We will use the notation $p_G$ for the quantity under investigation, $\P(d(x_0,x_4) < 4)$, 
		and for a vertex $v \in V(G)$, we define $p_k(v)$ to be the quantity 
		$\P(d(x_0,x_4) < 4 | x_k = v)$. 
	Note that $p_0(x_0) = p_G$ and $p_k(v) = 1$ if $v \in A_j$ for some $j < k$. 
	Finally, we call any step by the random walker that guarantees $d(x_0, x_4) < 4$ a ``stall.''

	We will break this proof into several statements.

\begin{claim}
\label{c1}
	Let $G'$ be the graph defined by removing all edges between $x_0$ and all but one vertex, $x_1$, 
	where $p_1(x_1) = \displaystyle \min_{v \in A_1} p_1(v)$. Then $p_{G'} \le p_G$. 
\end{claim}

\begin{proof} 
	Since $p_G < 1/s$, there must exist a vertex $v \in A_1$ with $p_1(v) < 1/s$. 
	Choose $x_1$ such that $p_1(x_1) = \displaystyle \min_{v \in A_1} p_1(v)$ and define $G'$ as in the statement of the claim. 
	Note that $p_{G'} = p_1(x_1) \le \displaystyle \frac{1}{a_1} \sum_{v \in A_1 \subseteq V(G)} p_1(v) = p_G$.
\end{proof}

\bigskip


\begin{claim}
\label{c2}
 
	Let $G''$ be the induced subgraph of $G'$ with $V(G'') = \displaystyle V(G') - \bigcup_{k > 4} A_k$ and with all edges removed except those that are between a vertex in $A_{k-1}$ and a vertex in $A_k$ for $k \in [4]$. Then $p_{G''} \le p_{G'}$.
\end{claim}  

\begin{proof}
	Let $\hat{G'}$ be the induced subgraph of $G'$ on the vertices $V(G') - \displaystyle \bigcup_{k>4} A_k$ for $k > 4$. 
	Then since $\P(x_t \in A_k) = 0$ when $t \le 4$ and $k \geq 5$ (so in particular, $p_{\hat{G'}}$ and $p_{G'}$ depend only on the first four steps of a random walk originating at $x_0$), we have that $p_{\hat{G'}} = p_{G'}$.

	Let $k \in [4]$ and let $v \in A_k$ be a vertex in $V(\hat{G'})$ with $N_{\hat{G'}}(v) \cap A_k \neq \emptyset$. 
	If no such vertex exists then $\hat{G'} = G''$. 
	Otherwise, let $\deg^-(v) = q, \deg^+(v) = r,$ and $|N_{\hat{G'}}(v) \cap A_k| = t >0$. $\displaystyle  p_k(v) \geq \frac{q+t}{q+r+t}$. Removing the $t$ vertices in $N_{\hat{G'}}(v) \cap A_k$ decreases $p_k(v)$ to $\displaystyle \frac{q}{q+r}$. Now let $G''$ be derived from $\hat{G'}$ by removing all edges except for those that are between $A_{k-1}$ and $A_k$. (In particular, this means that for all $k \in [4]$, for all $v \in A_k \cap V(G'')$, $N_{G''}(v) \cap A_k = \emptyset$.) This decreases $p_k(v)$ for all vertices $v$ with neighbors $w$ such that $d(x_0,v) = d(x_0,w)$ and does not change $p_k(v)$ for all vertices $v$ with no such neighbors. Since 
 
 $$ p_{G''} = \frac{1}{|N_{G''}(x_1)|} \sum_{u \in N_{G''}(x_1)} \frac{1}{|N_{G''}(u)|} \sum_{v \in N_{G''}(u)} \frac{1}{|N_{G''}(v)|} \sum_{w \in N_{G''}(v)} p_4(w)$$
 
 we have that $p_{G''} \le p_{\hat{G'}}$.
\end{proof}

\bigskip

In view of Claims~\ref{c1} and \ref{c2} above, we may assume that $G$ has the following properties: $N_G(x_0) = x_1$, the only edges in $G$ are between $A_{k-1}$ and $A_{k}$ for $k \in [4]$, and $A_k = \emptyset$ for all $k > 4$.

\begin{figure}[h!]
\centering
\includegraphics[scale=.8]{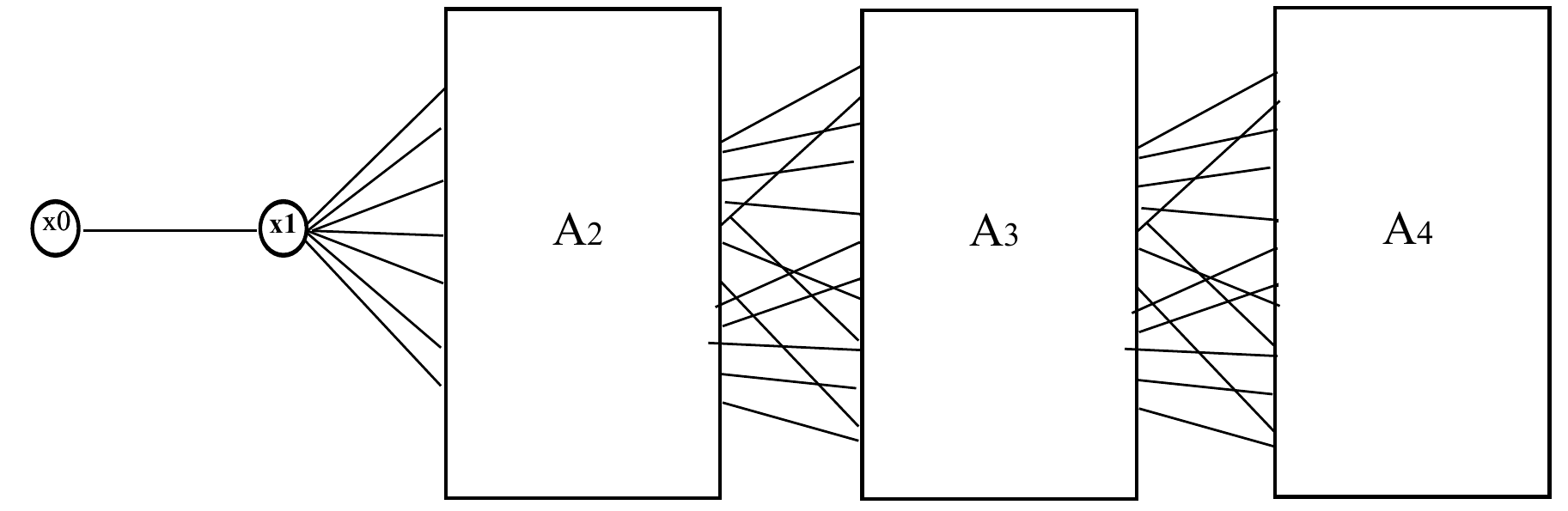}
\caption{Our (alleged) counterexample $G$}
\label{counterexample-4lemma}
\end{figure}

Now define $G_k \subseteq G$ to be the induced subgraph of $G$ on the vertices $A_k \cup A_{k+1}$ and let $e_k = |E(G_k)|$.

\begin{claim}
	\label{avgdeg}
 	$e_2 > s(s-1)$.
\end{claim}

 \begin{proof}
	Now we claim that the average degree of vertices in $A_2$ is greater than $s$: Let $\{d_i\}_1^{a_2}$ be the degrees of the vertices in $A_2$ and let $d = \displaystyle \sum_{i=1}^{a_2} d_i$. The vertex $x_2$ is chosen uniformly at random in $A_2$, and the probability of stalling at a vertex with degree $d_i$ is $\displaystyle \frac{1}{d_i}$. Therefore the probability of stalling at $A_2$ is 
$\displaystyle \frac{1}{a_2} \sum_1^{a_2} \frac{1}{d_i}$. We have 
$$1/s > \frac{1}{a_2} \sum_1^{a_2} \frac{1}{d_i} = \frac{1}{H(\{d_i\})} \geq a_2/d$$
(where $H(\{d_i\})$ is the harmonic mean of the $d_i$). Consequently, $d/a_2 > s$. Thus the average out-degree from $A_2$ is greater than $s-1$, which implies that there are more than $s(s-1)$ edges between $A_2$ and $A_3$.
\end{proof}

\begin{claim}
	Let $B$ be the subset of $A_2$ consisting of vertices with more than half of their outedges going to 	
	$C$, the subset of $A_3$ consisting of vertices with in-degree less than $n^{1/3}$. Let $b 
	= |B|$ and $c = |C|$. Then $b < \frac{1}{2}a_2$.
\end{claim}

\begin{proof} Define $e_B$ to be the number of edges with one endpoint in $B$ and the other in $A_3$. Note that $c \le a_3 < n-a_2 \le n - 4n^{2/3}$ and consequently the number of edges with one endpoint in $A_2$ and the other in $C$ is less than $ n^{4/3} - 4n$. Since more than half of the outedges of each vertex in $B$ terminate in a vertex in $C$, this says that $e_B < 2n^{4/3} - 8n$. 
%
%
%

Now assume, for sake of contradiction, that $b \geq \frac{1}{2} a_2$. Then $\P( x_2 \in B) > 1/2$ so we have
\begin{eqnarray*}
1/s &>& p_G \\
	&=& \P( d(x_0,x_4)< 4 | x_2 \in B) \P( x_2 \in B) + \P(  d(x_0,x_4)< 4 | x_2 \notin B) \P( x_2 \notin B) \\
	&>& \frac{1}{2} \P(  d(x_0,x_4)< 4 | x_2 \in B) 
\end{eqnarray*}

which says that $\P( d(x_0,x_4)< 4 | x_2 \in B) < \displaystyle \frac{2}{s}$.

Let $f = \displaystyle \sum_{i = 1}^b d_i$ where $\displaystyle \{d_i\}_{i=1}^b$ are the degrees of the vertices in $B$. For each vertex in $B$, $\displaystyle \P(d(x_0,x_4)<4| x_2 \in B; \deg(x_2) = d_i) = \frac{1}{d_i}$. Since $x_2$ is chosen uniformly at random, we have
	\begin{eqnarray*}
		\frac{2}{s} > \P(d(x_0,x_4)< 4 | x_2 \in B) &=& \frac{1}{b}\sum_{i = 1}^b \frac{1}{d_i} \\
					&=& \frac{1}{H(\{d_i\}_1^b)} \\
					&\geq & \frac{1}{(1/b)f} = \frac{b}{f} 
	\end{eqnarray*}

The average out-degree from $B$ is $\displaystyle \frac{f}{b} - 1$ and so we get $e_B \geq \displaystyle b \left(\frac{f}{b} - 1\right) \geq \frac{s}{2} \left(\frac{s}{2} - 1\right) = 4n^{4/3}-2n^{2/3}$. This is a contradiction since $2 n^{4/3} - 8n < 4n^{4/3} - 2n^{2/3}$.

Consequently, $b < \frac{1}{2}a_2$.
\end{proof}

\begin{claim}
	 The probability that $x_3 \in A_3 \backslash C$ (given $x_3 \in A_3$) is greater than $1/4$.
\end{claim}

\begin{proof} 
	If $x_2 \in A_2$ then with probability greater than $1/2$, $x_2 \in A_2 \backslash B$. 
	By definition, more than half of the out-edges of a vertex in $A_2 \backslash B$ terminate in $A_3 \backslash C$, and $x_3$ is chosen uniformly at random from the neighbors of $x_2$. 
	This yields 
	$$\P(x_3 \notin C | x_3 \in A_3) = \P(x_3 \notin C | x_2 \in B) \P(x_2 \in B) + \P(x_3 \notin C | x_2 \notin B)\P(x_2 \notin B)$$

	and therefore
	$$\P(x_3 \notin C | x_3 \in A_3) \geq \P(x_2 \notin B|x_2 \in A_2)\P(x_3 \notin C|x_2 \notin B)> (1/2)(1/2) = 1/4,$$
	as desired.
\end{proof}

	Note that $\P( d(x_0,x_4)< 4 | x_3 \in A_3 \backslash C)  \displaystyle = \frac{\deg^+(x_3)}{\deg(x_3)} > \frac{n^{1/3}}{n}$.
	Therefore the probability of stalling at step 3 is greater than $(1/4) \displaystyle \frac{n^{1/3}}{n} = 1/s$, yielding a contradiction. 
\end{proof}

	Let $j = 4i+1$ be such that the game is in Stage 3 at time $j$, and let $x_j, y_j$ be the positions of the drunk and cop, respectively, after both have moved (so that it is the drunk's turn).
	By Lemma~\ref{4lemma} we have that $d(x_j,x_{j-4}) < 4$ with probability at least $\frac{1}{4}n^{-2/3}$. Consequently, since the cop had $x_{j-4}$ as her target, we now have $d(y_j,x_j) < d(y_{j-4}, x_{j-4})$ (so the distance has decreased by at least 1) with probability at least $\frac{1}{4}n^{-2/3}$.  
	Let $Y_i$ be a random variable which equals the decrease in distance between time $4(i-1)$ and $4i$.
	$Y_i$ is 0 with probability less than $1-1/s$ and is $\geq 1$ with probability at least $1/s$.

	Consider the 0-1 random variable $X_i$ with $\P(X_i=1) = 1/s$ for all $i$ (note $\E[X_i] \le \E[Y_i]$ for all $i$).
	Let $S_n = X_1 + \dots + X_n$, for all $n \in \mathbb{N}$.
	Consider the random process $\{X_i: i \in \mathbb{N}\}$ with the stopping rule that says the process 
		terminates at time $\tau$ if $S_\tau = D_2 - 3$.
	By Wald's identity \cite{Wald}, $ \E[S_\tau] = \E[\tau] \E[X_i]$.
	Since $\E[S_\tau] = \E[D_2] - 3$, we have that the expected stopping time 
	$\E[\tau] = \displaystyle \frac{\E[D_2] - 3}{1/s}$.
	This is the expected number of retargetings needed to get $S_\tau = D_2 - 3$, so we have
	
	$$\E[T_3] \le 4\E[\tau] = 4s(\E[D_2] - 3) < 4((5 \log n)^{3/4} n^{1/4} - 3)(4n^{2/3}) $$

Stage 3 terminates when the distance between the cop and the drunk is less than four, and it is the cop's turn. 
In Stage 4, which terminates when the drunk is captured, the cop uses the greedy strategy, defined as follows.
Suppose that the strategy enters Stage 4 at time $t$, during which time the drunk is at vertex $x_t$ and the cop is about to move from vertex $y_{t-1}$. 
Then $d(x_t, y_{t-1}) \le 3$, and the cop moves such that $d(x_t, y_t) \le 2$. 
Now for any $r > t$, if the drunk moves such that $d(x_r, y_{r-1}) = 3$, the cop can choose $y_r$ to ensure that $d(x_r, y_r) = 2$. 
For each $r$, with probability at least $1/\Delta$, the drunk moves ``toward'' the cop---i.e., such that $d(x_r, y_{r-1}) = 1$; 
if that happens, the cop can choose $y_r = x_r$, capturing the drunk.
This takes at most $\Delta$ expected moves, so $\E[T_4] \le \Delta$ where $T_4$ is the expected time spent in Stage 4.


Adding together our results about the expected time to complete each of the four stages yields the following bound on the expected capture time:
\begin{eqnarray*}
\sum_{i=1}^4\E[T_i] &\le& diam(G) + \E[D_1] + 16n^{2/3} (\E[D_2] - 3) + \Delta \\
			&<& diam(G) + 1 + \sqrt{n(1 + 5\log n)} +  4((5 \log n)^{3/4} n^{1/4} - 3)(4n^{2/3}) + \Delta \\
			&=& diam(G) + \Delta + o(n)
\end{eqnarray*}

In fact, we can bound $diam(G) + \Delta$ with a bit of graph theory.

\begin{lemma}
	\label{diamG}
	For any graph $G$ with $|V(G)| = n$, $diam(G) + \Delta \le n  + 1$. 
\end{lemma}

\begin{proof}
Assume, for sake of contradiction, that there is a graph $G$ such that $diam(G) >  n - \Delta + 1$.
Let $u,v,w \in V(G)$ be (not necessarily distinct) vertices in $G$ such that $deg(v) = \Delta$ and $d(u,w) = d \geq n - \Delta + 2$.
Now we break this proof into two cases:

{\bf Case 1}: $v$ lies on a shortest path between $u$ and $w$. 

	Let $P_1$ be a shortest $u-w$ path containing $v$.
	At most two neighbors of $v$ may lie on $P_1$, so there are at least $\Delta - 2$ vertices not on $P_1$.
	Since the length of $P_1$ is at least $\geq n - \Delta + 2$, there are at least $n - \Delta + 3$ vertices in $P_1$.
	But now we have that $|V(G)| \geq \Delta - 2 + n - \Delta + 3 > n$, which is a contradiction.

{\bf Case 2}: $v$ is not on any shortest $u-w$ path.

	Let $P_2$ be a shortest $u-w$ path. 
	If more than 2 neighbors of $v$ are in $P_2$, then $v$ is also on a shortest $u-w$ path (let $x_1, x_2,$ and $x_3$ be the neighbors of $v$ on $P_2$, appearing in that order; then the section involving the three neighbors of $v$ could be replaced with $x_1 - v - x_3$ to create another shortest $u-w$ path). 
	Therefore there are at least $\Delta - 1$ vertices not on $P_2$ ($v$ and $\Delta-2$ of its neighbors), and at least $n - \Delta + 3$ vertices on this path. 
	So once again, $|V(G)| \geq \Delta - 1 + n - \Delta + 3 > n$, which is a contradiction.

Therefore $diam(G) \le n - \Delta + 1$ for all graphs $G$. 
\end{proof}

%

Therefore we have the following theorem about the expected capture time.

\begin{theorem}
On a connected, undirected, simple graph on $n$ vertices, a cop with the described four-stage strategy will capture a drunk in expected time $n + o(n)$. 
\end{theorem}

\section{Open Problems}

The reader may, for instance, have noticed that in the ``ladder to the basement'' example of Section~\ref{prelims}, we considered a cop who was not only greedy but also rather insistently foolish. What about the greedy cop who makes distance-minimizing decisions at random? The ``ladder to the basement'' graph is no longer a problem for her, 
(the expected capture time in this example is now less than $n$). Is it possible that the greedy algorithm with disputes settled by a random decision between choices is enough to guarantee time $n+o(n)$?

It is also possible that a deterministic greedy cop who breaks ties by considering her distance to vertices previously occupied by the drunk will capture in expected time at most $n + o(n)$.

An alternative greedy strategy, suggested by Andrew Beveridge~\cite{Bev}, concerns itself with minimizing the drunk's expected hitting time to the cop at every step. It would be interesting to see if this strategy also has expected capture time at most $n + o(n)$.

\section{Acknowledgments}

This work has benefited from conversations at Microsoft Research, in Redmond, Washington, with Omer Angel, Ander Holroyd, Russ Lyons, Yuval Peres, and David Wilson.

\end{document}